\documentclass[preprint]{elsarticle}
\usepackage{times}
\usepackage{hyperref}
\usepackage{amssymb, amsmath, amsthm}
\usepackage{multirow}
\usepackage{rotating}
\usepackage{indentfirst}
\usepackage{array}

\newtheorem{thm}{Theorem}
\newtheorem{corr}{Corollary}
\newtheorem{defn}{Definition}
\begin{document}
\begin{frontmatter}
\title{Bounds on the Expected Value of
Maximum Loss of Fractional Brownian Motion \footnote{This work is supported by scientific and TUBITAK project No. 110T674.}}
\author{Vardar, Ceren and Cakar, Hatice}
\address{Departmant of Mathematics TOBB Economy and Tecnology University\\ Sogutozu, Ankara, Turkey}

\begin{abstract}
In this study, it is theoretically proven that the expected value of maximum loss of fractional Brownian motion (fBm) up to time 1 with Hurst parameter $[1/2,1)$ is bounded above by $2/\sqrt{\pi}$ and below by $1/\sqrt{\pi}$. This result is generalized for fBm with $H\in[1/2,1)$ up to any fixed time, $t$. This also leads us to the bounds related to the distribution of maximum loss of fBm. As numerical study some lower bounds on the expected value of maximum loss of fBm up to time 1 are obtained by discretization. Simulation study is conducted with Cholesky method. Finally, comparison of the established bounds with simulation results is given.
\begin{keyword}
Cholesky Method \sep Hurst Parameter \sep Fractinal Brownian Motion \sep Maximum Loss \sep Sudakov-Fernique Inequality.
\end{keyword}
\end{abstract}

\end{frontmatter}
\section{Introduction.}
Let $H$ be a constant in the interval $(0,1)$. Fractional Brownian motion (fBm), $B_t^H$, ${t \geq 0}$ with Hurst parameter $H,$ is a continuous, centered Gaussian
process with covariance function
\begin{equation}E[B_t^H B_s^H]=\frac{1}{2}(t^{2H}+s^{2H}-|t-s|^{2H})\end{equation}
For $H=1/2$, fBm corresponds to standard Brownian motion.
\\
A standard fBm, $B_t^H$, $t \geq 0$ has the following properties, \cite{oksendal}:
\begin{itemize}
    \item $B_0^H=0$ and $E[B_t^H]=0$ for all $t\geq 0$.
    \item $B^H$ has homogenous increments, that is $B^H_{t+s}-B^H_s$ has the same law as $B_t^H,$ for $s,t\geq 0$.
    \item $B^H$ is a Gaussian process and $E[(B_t^H)^2]=t^{2H},~t\geq 0$, for all $H\in(0,1)$.
    \item $B^H$ has continuous trajectories.
\end{itemize}
\subsection{Correlation between two increments and long range dependence}
For $H=1/2$, the process $B_t^H$, ${t\geq 0}$ corresponds to a
standard Brownian motion, in which the increments are independent.
By the definition of fBm, we know the covariance between
$B^H(t+h)-B^H(t)$ with $s+h\leq t$ and $t-s=nh$ is
\begin{equation}\rho_H(n)=\frac{1}{2}h^{2H}[(n+1)^{2H}+(n-1)^{2H}-2n^{2H}].\end{equation}
We observe that two increments of the form $B^H(t+h)-B^H(t)$ and $B^H(t+2h)-B^H(t+h)$ are positively correlated for $H>1/2$, and they are negatively correlated for $H<1/2$, \cite{oksendal}.\\
\begin{defn}
A stationary squence $(X_n)_{n\in\mathbb{N}}$ exhibits long-range dependence if the autocovariance function $\rho(n):=cov(X_k,X_{k+n})$ satisfy $$\underset{n\rightarrow\infty}{lim}{\frac{\rho(n)}{cn^{-\alpha}}}=1$$ for some constant $c$ and $\alpha\in(0,1)$. In this case the dependence between $X_k$ and $X_{k+n}$ decays slowly as $n$ tends to infinity and $$\sum^\infty_{n=1}\rho(n)=\infty.$$
\end{defn}
Hence, it can be obtained that the increment $X_k=B_t^H(k)-B_t^H(k-1)$ of $B^H_t$ and $X_{k+n}=B_t^H(k+n)-B_t^H(k+n-1)$ of $B^H_t$ have the long-range dependence property for $H>1/2$ since $$\rho_H(n)=\frac{1}{2}[(n+1)^{2H}+(n-1)^{2H}-2n^{2H}]\approx H(2H-1)n^{2H-2}$$ as $n$ goes to infinity, \cite{oksendal}. In particular $$\underset{n\rightarrow\infty}{\lim}\frac{\rho_h(n)}{H(2H-1)n^{2H-2}}=1.$$
\subsection{Self similarity property}
FBm possesses the self-similarity property, that is for any constant
$c>0,$
\begin{equation}(B_{ct}^H)_{t\geq0}\stackrel{law}{=}(c^HB_t^H)_{t\geq0}.\end{equation}
\subsection{The model}
Fractional Brownian motion (fBm) is used in modelling many situations such as modelling temperature, characters of solar activity, values of the log returns of a stock, price of electricity, \cite{oksendal}. FBm with $H\in[\frac{1}{2},1)$ is used in finance due to its long-range dependence property, \cite{oksendal}. The Black-Scholes model for the values of returns of an asset using fBm is given as,
\begin{equation}
Y_t=Y_0exp{((r+\mu)t+\sigma B_t^H)}~~~~0\leq t \leq T,
\end{equation}
where $Y_0$ is initial value, $r$ constant interest rate, $\mu$ constant drift and $\sigma$ constant diffusion coefficient of fBm. Here, fBm is denoted by $B_t^H$, $t
\geq 0.$ Black-Scholes model can also be constructed with Markov processes such as Brownian motion. The advantage of modeling with fBm to Markov proccesses is its capability of exposing the dependence between returns. Real life data for a volatile
asset, display long-range dependence property. For this reason, modelling with fBm is more realistic compared to Markov processes.\\
\indent
In financial markets, investors would be interested in the maximum
possible loss as a measure of their highest possible risk. Some bounds on the
expected value of maximum loss of fBm already exists, \cite{vardarcaglar}. There are only bounds for fBm since exact results are not possible to find. Our aim is to improve these bounds.\\
\indent In this study, we provide a new and closer theoretical upper and lower bounds on the expected value of maximum loss. We obtain some
numerical lower bounds on the expected value of maximum loss using discretization and order statistics of multivariate normal variables. Simulation results are obtained by Cholesky method. As a conclusion, we provide comparison of the established bounds.
\subsection{Notation and Preliminaries}
Let $B_t^H$, $t\geq 0$ be a fBm defined on the probability space
$(\Omega,\mathcal{F},\textbf{P})$. Let \\ $S_t^H:=\underset{0\leq v\leq t}{\sup}B_v^H$ be the supremum, $I_t^H:=\underset{0\leq v\leq t}{\inf}B_v^H$ be the infimum of fBm, $R_t^H=S_t^H-I_t^H$ be the range, the difference between the supremum and the infimum, $M_t^{H}:=\underset{0\leq u \leq v\leq t}{\sup}(B_u^H-B_v^H)$ be the maximum loss of fBm. Some theoretical bounds found on the expected
value of supremum and the expected value of maximum loss of fBm,
\begin{itemize}
    \item An upper bound on the expected value of
    supremum was given as, \cite{Ceren}:
        \begin{equation}
        E(S_t^H)\leq \frac{\sqrt{2}}{\sqrt{\pi}}t^H.
        \end{equation}
    \item  Additionally, a lower bound on the expected value of infimum and an upper bound on the expected value of range are,\cite{vardarcaglar}:
        \begin{equation}
    -\frac{\sqrt{2}}{\sqrt{\pi}}t^H\leq E(I_t^H),~~~E(R_t^H)\leq \frac{2\sqrt{2}}{\sqrt{\pi}}t^H.
        \end{equation}
    \item In \cite{Norros} using Sudakov-Fernique inequality upper and lower bounds on the expected value of supremum of fBm were given as
    \begin{equation}
        \frac{t^H\sqrt{2}}{2\sqrt{\pi}} \leq E(S_t^H)\leq \frac{t^H\sqrt{2}}{\sqrt{\pi}}.
        \end{equation}
    \item Combining all the above and the fact that $E(S_t^H)$ equals $-E(I_t^H)$ provided bounds on the expected value of maximum loss of fBm up to time $t$, \cite{vardarcaglar}:
        \begin{equation}\label{bounds}
    \frac{\sqrt{2}}{2\sqrt{\pi}}t^H \leq E(M_t^{H})\leq \frac{2\sqrt{2}}{\sqrt{\pi}}t^H.
        \end{equation}
\end{itemize}
\section{Main Results}
In this section, we work on improving the bounds given in equation (\ref{bounds}). We prove that the expected value of maximum loss of fBm up to time $t$ is bounded above by $2t^H/\sqrt{\pi}$ which is a closer bound than $2\sqrt{2}t^H/\sqrt{\pi}$. Also, we prove that the expected value of maximum loss of fBm up to time $t$ is theoretically bounded below by $t^H/\sqrt{\pi}$ which is again a closer bound than $\sqrt{2}t^H/{2\sqrt{\pi}}$. Our second new result is to provide some numerical lower bounds on the expected value of maximum loss of fBm. Finally, we give the results of our simulation study in order to be able to compare the established bounds.
\subsection{Theoretical Results}
Here, we prove that the expected value of maximum loss of fBm up to time 1 with $H\in[1/2,1)$ is bounded above by $2/\sqrt{\pi}$ and it is bounded below by $1/\sqrt{\pi}$. These new bounds provide closer and improved values than the bounds given in equation (\ref{bounds}).
\begin{thm}\label{vardarcakar}
For fBm up to time 1, with Hurst parameter $1/2<H<1$,
\begin{equation*}
\frac{1}{\sqrt{\pi}}= E(\underset{0<s<1}{\sup}\sqrt{2}B_s^{1})\leq E(M_1^H)\leq E(\underset{0<s<1}{\sup}\sqrt{2}B_s^{1/2})=\frac{2}{\sqrt{\pi}}.
\end{equation*}
\end{thm}
\begin{proof}
Now, $\{B_u^H-B_v^H\}$ are centered Gaussian processes on seperable space $[0,1]$ and
\begin{eqnarray}\label{soltaraf}
E((B_{u}^{H}-B_{v}^{H})-(B_{u^{\prime }}^{H}-B_{v^{\prime
}}^{H}))^{2}&=&E((B_{u}^{H}-B_{v}^{H})^{2})\nonumber\\
&-&2E((B_{u}^{H}-B_{v}^{H})(B_{u^{\prime}}^{H}-B_{v^{\prime }}^{H}))\nonumber\\
&+&E((B_{u^{\prime }}^{H}-B_{v^{\prime}}^{H})^{2})\nonumber\\
&=&\left\vert u-v\right\vert ^{2H}+\left\vert u^{\prime }-v^{\prime}\right\vert ^{2H}+\left\vert u-u^{\prime }\right\vert ^{2H}\nonumber\\
&-&\left\vert u-v^{\prime }\right\vert ^{2H}-\left\vert v-u^{\prime }\right\vert^{2H}+\left\vert v-v^{\prime }\right\vert ^{2H}
\end{eqnarray}
For all $u,v\in[0,1]$ and $1/2<H<1$ note that,
\begin{equation}
{\left|u-v\right|}^{2H}\leq {\left|u-v\right|}=u+v-2min(u,v).
\end{equation}
Considering all 6 cases such as $0\leq u\leq v\leq u^{\prime}\leq v^{\prime} \leq 1$, $0\leq u\leq u^{\prime}\leq v\leq v^{\prime} \leq 1$, $0\leq u\leq u^{\prime}\leq v^{\prime}\leq v \leq 1$, $0\leq u^{\prime}\leq v^{\prime} \leq u\leq v \leq 1$, $0\leq u^{\prime}\leq u\leq v^{\prime}\leq v \leq 1$, $0\leq u^{\prime}\leq u\leq v\leq v^{\prime} \leq 1$ , one by one we see that,
\begin{eqnarray}
E((B_u^H-B_v^H)-(B_{u^{\prime}}^H-B_{v^{\prime}}^H))^2&=&\left\vert u-v\right\vert ^{2H}+\left\vert u^{\prime }-v^{\prime}\right\vert ^{2H}+\left\vert u-u^{\prime}\right\vert ^{2H}\nonumber\\
&-&\left\vert u-v^{\prime }\right\vert ^{2H}-\left\vert v-u^{\prime }\right\vert^{2H}+\left\vert v-v^{\prime }\right\vert ^{2H}\nonumber\\
&\leq&\left\vert u-v\right\vert ^{2H}+\left\vert u^{\prime }-v^{\prime}\right\vert ^{2H}\nonumber\\
&+&\left\vert u-u^{\prime}\right\vert ^{2H}+\left\vert v-v^{\prime }\right\vert ^{2H}\nonumber\\
&\leq&\left\vert u-v\right\vert+\left\vert u^{\prime }-v^{\prime}\right\vert\nonumber\\
&+&\left\vert u-u^{\prime}\right\vert+\left\vert v-v^{\prime }\right\vert\nonumber\\
&\leq& 2\max(u,v,u^{\prime},v^{\prime})+2\min(u,v,u^{\prime},v^{\prime})\nonumber\\
&-&4min(\min(u,v,u^{\prime},v^{\prime}),\max(u,v,u^{\prime},v^{\prime}))\nonumber\\
\end{eqnarray}
And, for all $u,v\in[0,1]$, note that
\begin{eqnarray}
E(\sqrt{2}B_u^{1/2}-\sqrt{2}B_v^{1/2})^2&=&E(2(B_u^{1/2})^2)+E(2(B_v^{1/2})^2)-2E(2B_u^{1/2}B_v^{1/2})\nonumber\\
&=&2\left|u-v\right|=2u+2v-4min(u,v)
\end{eqnarray}
Therefore, we have
\begin{eqnarray*}
E((B_u^H-B_v^H)-(B_{u^{\prime}}^H-B_{v^{\prime}}^H))^2 &\leq& 2\max(u,v,u^{\prime},v^{\prime})+2\min(u,v,u^{\prime},v^{\prime})\nonumber\\
&-&4min(\min(u,v,u^{\prime},v^{\prime}),\max(u,v,u^{\prime},v^{\prime}))\nonumber\\
&=&E(\sqrt{2}B_{\max\{u,v,u^{\prime},u^{\prime}\}}^{1/2}-\sqrt{2}B_{\min\{u,v,u^{\prime},u^{\prime}\}}^{1/2})^2\nonumber
\end{eqnarray*}
Hence, by the Sudakov-Fernique inequality given in \cite{Adler},
\begin{equation}
E(M_1^H)\leq E(\underset{0<s<1}{sup}\sqrt{2}B_s^{1/2})=E(\sqrt{2}S_1^{1/2})=\frac{2}{\sqrt{\pi}}
\end{equation}
is satisfied.\\
\indent Next, we prove lower bound of the expected value of the maximum loss of fBm up time $1$ is $1/\sqrt{\pi}$. From equation (\ref{soltaraf}), we have
\begin{eqnarray*}
E((B_{u}^{H}-B_{v}^{H})-(B_{u^{\prime }}^{H}-B_{v^{\prime
}}^{H}))^{2}&=&\left\vert u-v\right\vert ^{2H}+\left\vert u^{\prime }-v^{\prime}\right\vert ^{2H}+\left\vert u-u^{\prime }\right\vert ^{2H}\nonumber\\
&-&\left\vert u-v^{\prime }\right\vert ^{2H}-\left\vert v-u^{\prime }\right\vert^{2H}+\left\vert v-v^{\prime }\right\vert ^{2H}
\end{eqnarray*}
Suppose $0\leq  u\leq v\leq u^{\prime}\leq v^{\prime}\leq 1$ is satisfied. We would like to show that
\begin{equation}\label{son}
E((B_u^H-B_v^H)-(B_{u^{\prime}}^H-B_{v^{\prime}}^H))^2\geq\left\vert u-v\right\vert ^{2H}+\left\vert u^{\prime }-v^{\prime}\right\vert ^{2H}
\end{equation}
Equivalently, our aim is to show
\begin{equation}
\left\vert u-u^{\prime }\right\vert ^{2H}+\left\vert v-v^{\prime }\right\vert ^{2H}-\left\vert u-v^{\prime }\right\vert ^{2H}-\left\vert v-u^{\prime }\right\vert ^{2H}\geq 0.
\end{equation}
Now,
\begin{eqnarray}
& &\left\vert u-u^{\prime }\right\vert ^{2H}+\left\vert v-v^{\prime }\right\vert ^{2H}-\left\vert u-v^{\prime }\right\vert ^{2H}-\left\vert v-u^{\prime }\right\vert ^{2H}\nonumber\\
&=&\left\vert u-u^{\prime }\right\vert ^{2H}+\left\vert v-v^{\prime }\right\vert ^{2H}-\left\vert u-u^{\prime }+u^{\prime }-v^{\prime }\right\vert ^{2H}-\left\vert v-u^{\prime }\right\vert ^{2H}\nonumber\\
&\geq& \left\vert v-v^{\prime }\right\vert ^{2H}-\left\vert u^{\prime }-v^{\prime }\right\vert ^{2H}-\left\vert v-u{\prime }\right\vert ^{2H}
\end{eqnarray}
by triangle inequality. Take $u^{\prime}-v=a$ and $v^{\prime}-u^{\prime}=b$. Then
\begin{equation}
\left\vert v-v^{\prime }\right\vert ^{2H}+\left\vert u^{\prime}-v^{\prime }\right\vert ^{2H}-\left\vert v-u^{\prime }\right\vert ^{2H}
=\left\vert a+b\right\vert^{2H}-\left\vert a \right\vert^{2H}-\left\vert b\right\vert^{2H}
\end{equation}
where $0<a<1$, $0<b<1$ and $H\in(1/2,1)$.\\
For $0<b<a<1$, one can obtain
\begin{eqnarray}\label{taylor}
& &(a+b)^{2H}-a^{2H}-b^{2H}\\
&=&a^{2H}(1+\frac{b}{a})-a^{2H}-b^{2H}\nonumber\\
&=&a^{2H}(1+2H\frac{b}{a}+\frac{2H(2H-1)}{2}(1+c)^{2H-2}(\frac{b}{a})^2)-a^{2H}-b^{2H}\nonumber\\
&=&b(2Ha^{2H-1}-b^{2H-1})+H(2H-1)(1+c)^{2H-2}a^{2H-2}b^2\nonumber
\end{eqnarray}
by using Taylor expansion. In equation (\ref{taylor}), $c\in(0,b/a)$, $1<2H<2$ and $0<b/a<1$. Therefore, we see that
\begin{equation}
(a+b)^{2H}-a^{2H}-b^{2H}=\left\vert v-v^{\prime }\right\vert ^{2H}-\left\vert u^{\prime }-v^{\prime }\right\vert ^{2H}-\left\vert v-u^{\prime }\right\vert ^{2H} \geq 0.
\end{equation}
Then
\begin{equation}
\left\vert u-u^{\prime }\right\vert ^{2H}+\left\vert v-v^{\prime }\right\vert ^{2H}-\left\vert u-v^{\prime }\right\vert ^{2H}-\left\vert v-u^{\prime }\right\vert ^{2H}\geq 0.
\end{equation}
For $0<a<b<1$, by symmetry in the calculations we get the same result. And for $0<a=b<1$, since
\begin{equation}
(a+b)^{2H}-a^{2H}-b^{2H}=(2^{2H}-2)a^{2H}\geq 0,
\end{equation}
we obtain the same result.\\
Hence, for case $0\leq  u\leq v\leq u^{\prime}\leq v^{\prime}\leq 1$, we see that;
\begin{eqnarray}
E((B_u^H-B_v^H)-(B_{u^{\prime}}^H-B_{v^{\prime}}^H))^2&\geq&\left\vert u-v\right\vert ^{2H}+\left\vert u^{\prime }-v^{\prime}\right\vert ^{2H}\nonumber\\
&\geq&2{\min}(\left\vert u-v\right\vert ^{2H},\left\vert u^{\prime }-v^{\prime}\right\vert ^{2H})\nonumber\\
&=&\left\{\begin{array}{cc} 2E(B_u^H-B_v^H)^2, & \left\vert
u-v\right\vert ^{2H}<\left\vert u^{\prime }-v^{\prime}\right\vert
^{2H}\\ 2E(B_{u^{\prime}}^H-B_{v^{\prime}}^H)^2, & otherwise
\end{array}\right\}\nonumber
\end{eqnarray}
Now repeating the same arguments in all 6 cases such as $0\leq u\leq v\leq u^{\prime}\leq v^{\prime} \leq 1$, $0\leq u\leq u^{\prime}\leq v\leq v^{\prime} \leq 1$, $0\leq u\leq u^{\prime}\leq v^{\prime}\leq v \leq 1$, $0\leq u^{\prime}\leq v^{\prime} \leq u\leq v \leq 1$, $0\leq u^{\prime}\leq u\leq v^{\prime}\leq v \leq 1$, $0\leq u^{\prime}\leq u\leq v\leq v^{\prime} \leq 1$, one by one we see that some inequalities are satisfied again by the symmetry of calculations, that is
\begin{eqnarray}
& &E((B_u^H-B_v^H)-(B_{u^{\prime}}^H-B_{v^{\prime}}^H))^2\nonumber\\
&\geq&2 \min(\left\vert u-v\right\vert^{2H},\left\vert u^{\prime }-v^{\prime}\right\vert^{2H},\left\vert u-v^{\prime}\right\vert^{2H},\left\vert u^{\prime }-v\right\vert^{2H},\left\vert v-v^{\prime}\right\vert^{2H},\left\vert u-u^{\prime}\right\vert^{2H})\nonumber\\
&\geq& 2 \min(\left\vert u-v\right\vert^2,\left\vert u^{\prime }-v^{\prime}\right\vert^2,\left\vert u-v^{\prime}\right\vert^2,\left\vert u^{\prime }-v\right\vert^2,\left\vert v-v^{\prime}\right\vert^2,\left\vert u-u^{\prime}\right\vert^2)\nonumber\\
&=&2{\left|s-t\right|}^2\nonumber\\
&=&E(\sqrt{2}B_s^1-\sqrt{2}B_t^1)^2.
\end{eqnarray}
Here, $\left\vert s-t \right\vert^2=\min(\left\vert u-v\right\vert^2,\left\vert u^{\prime }-v^{\prime}\right\vert^2,\left\vert u-v^{\prime}\right\vert^2,\left\vert u^{\prime }-v\right\vert^2,\left\vert v-v^{\prime}\right\vert^2,\left\vert u-u^{\prime}\right\vert^2).$\\
Hence, by the Sudakov-Fernique inequality
\begin{equation}
\frac{1}{\sqrt{\pi}}=E(\underset{0<s<1}{sup}\sqrt{2}B_s^{1})=E(\sqrt{2}S_1^{1})\leq E(M_1^H)
\end{equation} is proved.
\end{proof}
\begin{corr}
Consider fBm up to fixed time t with $1/2<H<1$,
\begin{equation*}
\frac{t^H}{\sqrt{\pi}}=E(\underset{0<s<t}{\sup}\sqrt{2}B_s^{1})\leq E(M_t^H)\leq E(\underset{0<s<t}{\sup}\sqrt{2}B_s^{1/2})=\frac{2t^H}{\sqrt{\pi}}
\end{equation*}
and
$$P(M_t^H>x)\leq \frac{2t^{H}}{x\sqrt{\pi}}.$$
\end{corr}
\begin{proof}
By the self similarity property of fBm
\begin{eqnarray}
E(M_t^H)&=&E(\underset{0<u<v<t}{\sup}(B_u^H-B_v^H))\nonumber\\
&=&E(\underset{0<u<v<1}{\sup}({t^H}(B_u^H-B_v^H)))\nonumber\\
&=&t^HE(M_1^H).
\end{eqnarray}
And, by Markov's inequality,
\begin{equation}
P(M_t^H>x)\leq \frac{E(M_t^H)}{x} \leq \frac{2t^{H}}{x\sqrt{\pi}}.
\end{equation}
\end{proof}
\subsection{Numerical Results}
Let $B_t^H$, $t\geq 0$ be a fBm defined on the probability space
$(\Omega,\mathcal{F},\textbf{P}).$ We consider fBm with Hurst
parameter $1/2<H<1,$ which is on the interval $(0,t)$. We divide time $t$ into $n$ parts and obtain the increments $\{B_{i\delta}^H-B_{(i-j)\delta}^H:~i:1,2...,n;j=0,1,...,i\}$ with $\delta=t/n$, where each one of them is centered Gaussian variable with different variances. After such discretization, we have the following inequality,
\begin{equation}\max_{1\leq i\leq n,0\leq j \leq
i}(B_{i\delta}^H-B_{(i-j)\delta}^H)\leq M_t^H.
\end{equation}
\textbf{Covariance between increments}
The increments given above are centered Gaussian variables,
and they are not independent, the covariance between them is
\begin{eqnarray*}
E[(B_{i\delta}^H-B_{(i-j)\delta}^H)(B_{k\delta}^H-B_{(k-l)\delta}^H)]&=&E[B_{i\delta}^HB_{k\delta}^H]-E[B_{i\delta}^HB_{(k-l)\delta}^H]\\
&-&E[B_{(i-j)\delta}^HB_{k\delta}^H]+E[B_{(i-j)\delta}^HB_{(k-l)\delta}^H]\\
&=&\frac{1}{2}[|i\delta-(k-l)\delta|^{2H}+|(i-j)\delta-k\delta|^{2H}]\\
&+&\frac{1}{2}[-|i\delta-k\delta|^{2H}-|(i-j)\delta-(k-l)\delta|^{2H}]
\end{eqnarray*}
where $i=1,2...,n; j=0,1...,i; k=1,2...,n; l=0,1...,k$.\\
In the literature, there are many studies on the expected value and the
distribution of maximum of dependent Gaussian variables. The history
of this issue starts with Tippett, \cite{tippett}, Teichroew, \cite{Teichroew}, Clark and Williams, \cite{cm} and Bose and Gupta, \cite{bg}. Owen and Steck, \cite{os}
considered dependent but identically distributed Gaussian variables.
Clark, \cite{clark} obtained first four moments for two dependent and
differently distributed Gaussian variables, however these results
were only for two variables. Ross, \cite{ross2003} gives an upper bound for
dependent and differently distributed Gaussian variable, however in
our case we need a lower bound. Lai and Robinson, \cite{1976} found another
bound for dependent but identically distributed variables.
Brodtkorb, \cite{Brodtkorb} presented an approach for variables with singular
covariance matrix, however this approach also is not proper for our
study. There is no exact, theoretical solution for this problem.
Ross, \cite{mainpaper} presented numerical methods to obtain lower and upper
bounds for the expected value of maximum of dependent and
differently distributed (DDD) Gaussian variables using Theorem \ref{vitale} given by Vitale, \cite{vita}.
\begin{thm}\label{vitale}
Let $W_i,X_i,Y_i,~i=1,2...,n$ be DDD centered Gaussian variables.
And for all $i$ and $j$
$$E((W_i-W_j)^2)\leq E((X_i-X_j)^2)\leq E((Y_i-Y_j)^2)$$
is satisfied.
Then, for arbitrary constants $m_i,~ i=1,2...,n,$ $$E[\max_{i}{W_i}+m_i]\leq
E[\max_{i}{X_i}+m_i]\leq E[\max_{i}{Y_i}+m_i].$$\\
\end{thm}
In our study, we implement Ross, \cite{mainpaper} method in order to obtain lower bound for the
expected value of maximum loss of fBm. Ross, \cite{mainpaper} numerically finds lower and upper bounds on the expected value of maximum of DDD Gaussian variables which can't be calculated explicity. Using covariances of DDD variables, perfectly dependent and differently distributed (PDDD) Gaussian variables and independent and differently distributed (IDD) Gaussian variables which
satisfy Theorem \ref{vitale} are found. Lastly, by Theorem \ref{vitale},
\begin{eqnarray}
E[max(PDDD_{1})]&\leq& E[max(DDD)]\leq E[max(PDDD_{2})]\\
E[max(IDD_{1})]&\leq &E[max(DDD)]\leq E[max(IDD_{2})]
\end{eqnarray} are obtained.\\
In our study,
\begin{enumerate}
 \item Let $\{W_i,~~i=1,2...,n\}$ be PDDD Gaussian variables and $\{B_{i\delta}^H-B_{(i-j)\delta}^H,~~i=1,2...,n, j=0,1...,i\}$ are DDD Gaussian variables.
Suppose
$$ E((W_i-W_j)^2)\leq E[((B_{i\delta}^H-B_{(i-j)\delta}^H)-(B_{k\delta}^H-B_{(k-l)\delta}^H))^2]$$
holds.
Then by Theorem \ref{vitale} given above,
$$E(\underset{i}{\max}W_i)\leq E(\underset{i}{\max}(B_{i\delta}^H-B_{(i-j)\delta}^H))$$
One can construct many $W_i$'s which satisfy the above inequality.
Therefore, we try to find the maximum value of $E(\underset{i}{\max}W_i).$ The
calculation of this is given as
    $E(\underset{i}{\max}W_i)=E[E[\underset{i}{\max}W_i|Z]]=\int^{\infty}_{-\infty}h(z)\phi(z)dz$, where $h(z)=E(\underset{i}{\max}W_i|Z=z)$ and $\phi(z)$ is the probability density function of standart normal variable.
\item Let $\{Y_i,~ i=1,2...,n\}$ be IDD Gaussian variables and $\{B_{i\delta}^H-B_{(i-j)\delta}^H,~i=1,2...,n, j=0,1...,i\}$ are DDD Gaussian variables.
 Suppose following inequality holds,
$$E((Y_i-Y_j)^2)\leq
E[((B_{i\delta}^H-B_{(i-j)\delta}^H)-(B_{k\delta}^H-B_{(k-l)\delta}^H))^2]$$
Again, by Theorem \ref{vitale}, we have
$$E(\max_{i}Y_i)\leq E(\max_{i}(B_{i\delta}^H-B_{(i-j)\delta}^H)).$$
Hence, same algorithm as the algorithm given in 1. is applied. The
only difference is, in calculating $E(\underset{i}{max}Y_i)$ that is,
$$E(\max_{i}Y_i)=\int^{\infty}_{0}(1-\prod^{n}_{i=1}Pr(Y_i\leq y)-\prod^{n}_{i=1}Pr(Y_i\leq
-y))dy$$
\end{enumerate}
Following this algorithm, we have conducted a numerical study using MATLAB. Table \ref{rossmax}, shows numerical lower bounds we have found on the expected value of maximum loss of fBm up to time $1$, with Hurst parameter $H\in(1/2,1)$.
\begin{table}[htbp]
    \centering
    \setlength{\extrarowheight}{6pt}
        \begin{tabular}{c c c c c c c c}

            & & $n=5$ & $n=10$ &    $n=20$ &    $n=30$ &    $n=40$ &    $n=50$ \\
            \hline

            \multirow{2}{*}{\begin{sideways}$H=0.5$\end{sideways}}&IDD Lower Bound&$0.6350$ & $0.6551$  &$0.5753$&  $0.5171$&   $0.4753$&   $0.4434$\\
              &PDDD Lower Bound&$0.2847$ &  $0.1621$&   $0.1294$    &$0.1056$&  $0.0903$&   $0.0872$\\
                \hline
            \multirow{2}{*}{\begin{sideways}$H=0.6$\end{sideways}}&IDD Lower Bound& $0.5230$&$0.4855$   &$0.3954$   &$0.3407$   &$0.3040$&  $0.2772$\\
                &PDDD Lower Bound&$0.2429$  &$0.1250$   &$0.0956$&  $0.0579$&   $0.0624$&   $0.0535$\\
                \hline
            \multirow{2}{*}{\begin{sideways}$H=0.7$\end{sideways}}&IDD Lower Bound&$0.4309$ &   $0.3550$    &$0.2645$   &$0.2179$   &$0.1886$&  $0.1680$\\
                & PDDD Lower Bound&$0.2037$ &   $0.1055$&   $0.0695$&   $0.0491$&   $0.0425$&   $0.0305$\\
            \hline
            \multirow{2}{*}{\begin{sideways}$H=0.8$\end{sideways}}&IDD Lower Bound& $0.3578$&   $0.2688$&   $0.1830$&   $0.1432$&   $0.1196$&   $0.1038$\\
                &PDDD Lower Bound& $0.1682$&    $0.0859$    &$0.0478$   &$0.0357$   &$0.0274$&$0.0192$\\
        \hline
            \multirow{2}{*}{\begin{sideways}$H=0.9$\end{sideways}}&IDD Lower Bound&$0.2760$ &   $0.2010$&   $0.1324$&   $0.1005$    &$0.0819$   &$0.0696$\\
                &PDDD Lower Bound&$0.1420$ & $0.0674$&  $0.0348$    &$0.0201$&  $0.0200$&   $0.0127$
        \end{tabular}
    \caption{Numerical lower bounds on the expected value of maximum loss of fBm up to time $1$}\label{rossmax}
\end{table}

\subsection{Simulation Results}
In our study, we generate fBm with $1/2<H<1$ up to time 10000 with 10000 simulations using Cholesky method. This method is an exact simulation method for fBm, \cite{Deiker}. From this, we collect 10000 values of maximum loss of fBm up to time 10000. We find their average, which gives us the expected value of maximum loss of fBm up to time 10000. Afterwards, using the self similarity property of fBm we calculate the expected value of maximum loss of fBm up to time 1, as follows
\begin{equation}\label{selfmax}
E(M_1^H)=\frac{E(M_t^H)}{t^H}
\end{equation}
We repeat this for different values of Hurst parameter. The results are given in Table \ref{tablomax1}.
\begin{table}[htbp]

\centering
        \begin{tabular}{c c c}
            Hurst Parameter& Cholesky Simulation results \\ \hline
$0.5$   &$1.239$\\
$0.6$   &$1.00721$\\
$0.7$   &$0.82509$\\
$0.8$   &$0.69865$\\
$0.9$   &$0.61016$\\
        \end{tabular}
\caption{Simulation results for the expected value of maximum loss of fBm up to time 1}\label{tablomax1}
\end{table}
\section{Conclusion}
Any information on maximum possible loss as a measure of risk would be important for investors in order to hedge it or manage it. Here, we provide an upper bound and a lower bound on the expected value of maximum loss of fBm up to time $1$ and up to fixed time, $t$. By using Markov's inequality, we find an upper bound on the distribution of maximum loss. Later, we obtain numerical lower bounds on the expected value of maximum loss of fBm up to time $1$ with $1/2<H<1$. Table \ref{maxtumsonuclar} shows all these results on the expected value of the maximum loss of fBm up to time 1.
\begin{table}[htbp]
    \centering
    \setlength{\extrarowheight}{6pt}
    \scalebox{0.5}{
        \begin{tabular}{c c c c c c c c c c c c c c c c c rrrrrrr}
            \large{Hurst}  & \large{Cholesky Simulation} &\large{Caglar and Vardar}& \large{Caglar and Vardar}& \large{Vardar and Cakar}& \large{Vardar and Cakar}\\

\large{Parameter}& \large{Results} &\large{Lower Bound} &\large{Upper Bound} & \large{Lower Bound} &\large{Upper Bound}\\
\hline
\large{$0.5$}   &\large{$1.239$}&   \large{$0.39904$}&  \large{$1.59617$}&\large{$0.56418$}&    \large{$1.12866$}\\
\large{$0.6$}   &\large{$1.00721$}& \large{$0.39904$}&  \large{$1.59617$}&\large{$0.56418$}&    \large{$1.12866$}\\
\large{$0.7$}&\large{$0.82509$}&    \large{$0.39904$}&  \large{$1.59617$}&\large{$0.56418$}&    \large{$1.12866$}\\
\large{$0.8$}&  \large{$0.69865$}&  \large{$0.39904$}&  \large{$1.59617$}&\large{$0.56418$}&    \large{$1.12866$}\\
\large{$0.9$}   &\large{$0.61016$}& \large{$0.39904$}&  \large{$1.59617$}&\large{$0.56418$}&    \large{$1.12866$}

        \end{tabular}}

\centering
\scalebox{0.5}{
\begin{tabular}{c c c c c c c c c c c c c c c c c c}

         \large{Hurst} &     \multicolumn{6}{c}{\large{IDD Lower Bound}}& \multicolumn{6}{c}{\large{PDDD Lower Bound}} \\

    \large{Parameter}&\large{$n=5$}& \large{$n=10$}&\large{$n=20$}& \large{$n=30$}& \large{$n=40$}& \large{$n=50$}& \large{$n=5$}&\large{$n=10$}&   \large{$n=20$}& \large{$n=30$}&\large{$n=40$}&  \large{$n=50$}\\
\hline
\large{$0.5$}& \large{$0.6350$}&    \large{$0.6551$}&   \large{$0.5753$}&   \large{$0.5171$}&   \large{$0.4753$}&   \large{$0.4434$}&   \large{$0.2847$}    & \large{$0.1621$}& \large{$0.1294$}&   \large{$0.1056$}&   \large{$0.0903$}&   \large{$0.0872$}\\
\large{$0.6$}& \large{$0.5230$}&    \large{$0.4855$}&   \large{$0.3954$}&   \large{$0.3407$}&   \large{$0.3040$}&   \large{$0.2772$}&   \large{$0.2429$}    &\large{$0.1250$}&  \large{$0.0956$}&   \large{$0.0579$}&   \large{$0.0624$}&   \large{$0.0535$}\\
\large{$0.7$}& \large{$0.4309$}&    \large{$0.3550$}&   \large{$0.2645$}&   \large{$0.2179$}&   \large{$0.1886$}&   \large{$0.1680$}&   \large{$0.2037$}  &\large{$0.1055$}&    \large{$0.0695$}&   \large{$0.0491$}&   \large{$0.0425$}&   \large{$0.0305$}\\
\large{$0.8$}& \large{$0.3578$}&  \large{$0.2688$}& \large{$0.1830$}&   \large{$0.1432$}&   \large{$0.1196$}&   \large{$0.1038$}& \large{$0.1682$}  &\large{$0.0859$}&  \large{$0.0478$}&   \large{$0.0357$}&   \large{$0.0274$}&   \large{$0.0192$}\\
\large{$0.9$}& \large{$0.2760$}&  \large{$0.2010$}& \large{$0.1324$}&   \large{$0.1005$}&   \large{$0.0819$}&   \large{$0.0696$}&   \large{$0.1420$}  &\large{$0.0674$}&    \large{$0.0348$}&   \large{$0.0201$}&   \large{$0.0200$}&   \large{$0.0127$}

        \end{tabular}}
        \caption{Results for the expected value of the maximum loss of fBm up to time one}\label{maxtumsonuclar}
            \end{table}
\newpage
Here, Caglar and Vardar upper and lower bounds are the theoretical bounds given in \cite{vardarcaglar}. Comparing the simulation results with the values given in Table \ref{maxtumsonuclar}, we see that Vardar and Cakar theoretical upper bound and lower bound which were obtained in Theorem \ref{vardarcakar}, for $1/2<H<1$, provides the closest upper and lower bounds among all the bounds.
\section{Acknowledgement}
We would like to thank Mine Caglar from Koc University, Istanbul for her guidance in finding the lower bound given in Theorem \ref{vardarcakar}.

\end{document}